\documentclass[a4paper,10pt]{article}
\usepackage[latin1]{inputenc}
\usepackage[T1]{fontenc}
\usepackage{amsmath}
\usepackage{amsfonts}
\usepackage{amstext}
\usepackage{amsthm}
\usepackage[all]{xy}
\usepackage{geometry}
\usepackage{graphicx}
\newtheorem{theorem}{Theorem}[section]
\newtheorem{lemma}[theorem]{Lemma}

\newtheorem{rem}[theorem]{Remark}
\DeclareMathOperator{\supp}{supp}
\DeclareMathOperator{\im}{im}
\DeclareMathOperator{\ind}{ind}

\DeclareMathOperator{\sfl}{sf}

\DeclareMathOperator{\pr}{pr}
\DeclareMathOperator{\sgn}{sgn}

\title{A K-Theoretic Proof of the Morse Index Theorem in Semi-Riemannian Geometry}
\author{Nils Waterstraat}

\begin{document}
\date{}
\maketitle

\begin{abstract}
We give a short proof of the Morse index theorem for geodesics in semi-Riemannian manifolds \cite{Pejsachowicz} by using $K$-theory. This makes the Morse index theorem reminiscent of the Atiyah-Singer index theorem for families of selfadjoint elliptic operators \cite{AtiyahPatodi}.
\end{abstract}

\section{Introduction}
Let $(M,g)$ be a semi-Riemannian manifold of index $\nu$, let $p,q\in M$ be fixed points and let $I=[0,1]$ be the unit interval. Let $H^1_{p,q}(I,M)$ be the set of all paths $\gamma:I\rightarrow M$ that join $p$ and $q$ and are of Sobolev regularity $H^{1,2}$. $H^1_{p,q}(I,M)$ can be given the structure of a Hilbert manifold such that the action functional
\begin{align}\label{Action}
\mathcal{A}:H^1_{p,q}(I,M)\rightarrow\mathbb{R},\quad \mathcal{A}(\gamma)=\int^1_0{g(\dot\gamma,\dot\gamma)dt}
\end{align}
is a smooth function (cf.\cite{DiplomIch}). Moreover, the critical points of $\mathcal{A}$ are precisely the geodesics joining $p$ and $q$. Given such a geodesic $\gamma:I\rightarrow M$, a vector field $\xi\in\Gamma(\gamma)$ along $\gamma$ is called Jacobi field if it satisfies the differential equation
\begin{align}\label{JacobiEquation}
\frac{\nabla^2}{dx^2}\xi+R(\dot\gamma,\xi)\dot\gamma=0,
\end{align}
where $R$ denotes the curvature tensor of $(M,g)$ and $\frac{\nabla}{dx}$ the covariant derivative along $\gamma$. An instant $t\in I$ is said to be conjugate if there exists a nontrivial Jacobi field $\xi\in\Gamma(\gamma)$ such that $\xi(0)$ and $\xi(t)$ vanish.\\
Let us consider solely the Riemannian case $\nu=0$ for a moment. If $p$ and $q$ are chosen appropriately, one can show that $\mathcal{A}$ is a Morse function; i.e., at each critical point the Hessian is nondegenerate and has finite Morse index. Therefore the topology of $H^1_{pq}(I,M)$ (being homotopy equivalent to the based loop space of $M$) can be studied by using the machinery of Morse theory (cf. \cite{MilnorMorse}). Here a particularly important tool is given by the Morse index theorem, which states that the Morse index at a critical point $\gamma$ can be computed as
\begin{align}\label{mit}
\ind(\gamma)=\sum_{x\in I}{m(x)},
\end{align}
where $m(x)$ is the multiplicity of $x$ as conjugate instant, i.e. the dimension of the vector space spanned by Jacobi-fields along $\gamma$ that vanish at $0$ and $x$. Using this result, the computation of the Morse index of the Hessian at a given critical point reduces to solve a second order ordinary linear differential equation!\\
However, in arbitrary semi-Riemannian manifolds phenomena may occur that exclude an equality like \eqref{mit}. To be more precise, if $\nu\neq 0$, geodesics never have a finite Morse index and conjugate instants may accumulate such that the sum on the right hand side of \eqref{mit} is not even meaningful in general. Moreover, conjugate instants can disappear under arbitrary small perturbations of the geodesics, which is also in contrast to the stability of the Morse index in the cases where it exists. Hence, in order to find an index theorem valid for geodesics in arbitrary semi-Riemannian manifolds which generalizes the known results, one has to find some kind of renormalized Morse index and a way of generalized counting of conjugate points.\\
We do not want to present the full history of this issue, but mention that the first breakthrough was obtained by Helfer in \cite{HELFER}. Afterwards Piccione, Tausk and others developed his ideas further in a series of papers (cf. e.g. \cite{PiccioneMITinSRG} and the references given there). In this paper we focus on a more recent result by Musso, Pejsachowicz and Portaluri \cite{Pejsachowicz}. They gave for each side of the classical Morse index theorem a further substitute extending it to the general semi-Riemannian case and proved their equality as well as that their indices coincide with the former ones. These integers are called the spectral index and the conjugate index of the given geodesic $\gamma$. The spectral index, which generalizes the classical Morse index, is by definition the spectral flow of a path of bounded selfadjoint Fredholm operators associated by the Riesz representation theorem to the Hessian of the smooth function \eqref{Action} at the critical point $\gamma$. The conjugate index, which generalizes the classical counting of conjugate points, is the winding number of a closed curve surrounding $0\in\mathbb{C}$ which can be obtained from the differential equation \eqref{JacobiEquation}. The proof of their equality in \cite{Pejsachowicz} uses functional analytic methods. Our aim is to present a different proof which uses topological methods and can be roughly described as follows: The starting point is the original construction of spectral flow in \cite{AtiyahPatodi}, where to each closed path of selfadjoint Fredholm operators a virtual bundle in $K^{-1}(S^1)\cong\tilde{K}(S^2)$ is constructed such that the spectral flow of the path is just the integer obtained from the well-known isomorphism $c_1:K^{-1}(S^1)\rightarrow\mathbb{Z}$ given by the first Chern number. We modify this result slightly, obtaining an assignment of a virtual bundle in $K_c(\mathbb{C})$ for a special class of paths of selfadjoint Fredholm operators with invertible ends. Again the spectral flow can be computed by the isomorphism given by first Chern number $c_1:K_c(\mathbb{C})\rightarrow\mathbb{Z}$, and the second important observation is the well-known result that this integer can be calculated as the winding number of the clutching function of the given element in $K_c(\mathbb{C})\cong\tilde{K}(S^2)$. Hence the strategy of our proof is to consider the element in $K_c(\mathbb{C})\cong\tilde{K}(S^2)$ whose first Chern number is the spectral index of the given geodesic and to deform its clutching function such that the associated winding number turns out to be the conjugate index.\\
The paper is structured as follows: In the second section we give a short summary of the main result in \cite{Pejsachowicz}. The third section is devoted to the proof of this result and is decomposed in three subsections. In the first subsection we summarize briefly the definition of K-theory with compact supports and the computation of the first Chern number of elements in $K_c(\mathbb{C})$. The second subsection deals with spectral flow and is of independent interest. In the last subsection we finally present the alternative proof of the Morse index theorem in semi-Riemannian geometry \cite{Pejsachowicz}.\\
The author wishes to express his gratitude to one of the authors of \cite{Pejsachowicz}, Jacobo Pejsachowicz (Politecnico di Torino), for suggesting the problem and his interest in this work. Moreover, he wants to thank Jacobo Pejsachowicz and Thomas Schick (University of G\"ottingen) for several helpful suggestions improving the presentation of the paper. Finally, the author wishes to thank the Research Training Group 1493, ``Mathematical Structures in Modern Quantum Physics'', for financial support.

\section{The Morse Index Theorem in semi-Riemannian Geometry}
It is a well known result that the Hessian $Hess_{\gamma}(\mathcal{A}_{pq}):T_\gamma H^1_{pq}(I,M)\times T_\gamma H^1_{pq}(I,M)\rightarrow\mathbb{R}$ of the functional \eqref{Action} at a critical point $\gamma\in H^1_{p,q}(I,M)$ is given by
\begin{align*}
Hess_{\gamma}(\mathcal{A}_{pq})(\xi,\eta)&=\int^1_0{g_{\gamma(x)}\left(\frac{\nabla}{dx}\xi(x),\frac{\nabla}{dx}\eta(x)\right)dx}-\int^1_0{g_{\gamma(x)}(R(\gamma'(x),\xi(x))\gamma'(x),\eta(x))dx},
\end{align*}
where $\xi,\eta\in T_\gamma H^1_{p,q}(I,M)=\{\xi\in H^1(I,TM): \xi(x)\in T_{\gamma(x)}M\;\forall x\in I,\xi(0)=0,\xi(1)=0\}$.\\
These definitions suffice to state the Riemannian version of the Morse index theorem as the well-known equality of the finite Morse index of the Hessian and the number of conjugate points along the geodesic counted with multiplicity as already mentioned in the introduction. To overcome the rising difficulties in the general semi-Riemannian case, as in \cite{Pejsachowicz}, some variations are necessary.\\
Let $\gamma$ be a geodesic joining two points $p$ and $q$ such that $1$ is not a conjugate instant. $\gamma$ induces a path $\Gamma:I\rightarrow H^1(I,M)$ by $\gamma_t(x)=\gamma(t\cdot x)$ and $\gamma_t$ is a critical point of the restricted action functional $\mathcal{A}_t:H^1_{p,\gamma(t)}(I,M)\rightarrow\mathbb{R}$. 
This yields a family of quadratic forms 
\begin{align*}
hess_t(\mathcal{A}):T_{\gamma_t}H^1_{p,\gamma(t)}(I,M)\rightarrow\mathbb{R}
\end{align*}
associated to the Hessians that can be transformed into a smooth path of quadratic forms on $H^1_0([0,1],\mathbb{R}^n)$ by using local coordinates induced by a parallel frame $\{e^1,\ldots,e^n\}$ along $\gamma$, where
\begin{align*}
g(e^i(x),e^i(x))=\begin{cases}
1,\;i\leq n-\nu\\
-1,\;i>n-\nu
                 \end{cases},\quad g(e^i(x),e^j(x))=0,\; i\neq j,\; x\in I.
\end{align*}
The resulting path is given by
\begin{align*}
\begin{split}
&q_t:H^1_0(I,\mathbb{R}^n)\rightarrow\mathbb{R},\\ 
q_t(u)=&\int^1_0{\langle Ju'(x),u'(x)\rangle dx}-\int^1_0{\langle S_t(x)u(x),u(x)\rangle dx},
\end{split}
\end{align*}
where
\begin{align*}
J=\begin{pmatrix}
id_{n-\nu}&0\\
0&-id_{\nu}
\end{pmatrix}, 
\end{align*}
$S_t(x)=t^2\cdot S(t\cdot x)$, $t,x\in[0,1]$, and $S(x)=\{S_{ij}(x)\}$, $S_{ij}(x)=g(R(\gamma'(x),e^j(x))\gamma'(x),e^i(x))$.
It is easy to see that $S(x)$ is symmetric for all $x\in I$.\\
Using the Riesz representation theorem, we obtain a path $L:I\rightarrow\mathcal{L}(H^1_0(I,\mathbb{R}^n))$ of bounded selfadjoint operators such that
\begin{align}\label{RIESZDARSTELLUNG-2}
q_t(u)=\langle L_tu,u\rangle_{H^1_0(I,\mathbb{R}^n)}\quad \text{ for all}\; u\in H^1_0(I,\mathbb{R}^n)
\end{align}
and $L_0,L_1\in GL(H^1_0(I,\mathbb{R}^n))$, where the invertibility of $L_1$ follows from the assumption that $1$ is not a conjugate instant and integration by parts. In \cite[Proposition 3.1]{Pejsachowicz} it is shown that these operators actually are Fredholm. The \textbf{spectral index} of the geodesic $\gamma$ is defined as $\mu_{spec}(\gamma)=-\sfl(L)$, where $\sfl$ denotes the spectral flow.\\
The other important modification of the classical quantities from the Morse index theorem is the counting of conjugate points along a geodesic by means of the winding number as already mentioned in the introduction. Using once again the orthonormal frame $\{e^1,\ldots,e^n\}$, the Jacobi equations \eqref{JacobiEquation} associated to the geodesics $\gamma_t$ are given by
\begin{align}\label{JacobiDGL}
Ju''(x)+S_t(x)u(x)=0.
\end{align}
Note at first that the corresponding differential operators $\mathcal{A}_tu=Ju''+S_tu$ are unbounded selfadjoint Fredholm operators on $L^2(I,\mathbb{R}^n)$ with domain $H^2(I,\mathbb{R}^n)\cap H^1_0(I,\mathbb{R}^n)$. Moreover, $t\in I$ is a conjugate instant if and only if $\ker\mathcal{A}_t\neq\{0\}$. Now let $S_t:=S_0$ for $t<0$, $S_t:=S_1$ for $t>1$ and let $b^1_z,\ldots,b^n_z:I\rightarrow\mathbb{C}^n$ be the solutions of the differential equations
\begin{align*}
Ju''+S_t(x)u+isu=0,\quad z=t+is\in\mathbb{C},
\end{align*}
such that $b^i_z(0)=0$ and $(b^i_z)'(0)=Je_i$, $i=1,\ldots,n$. We denote by $b_z(x)$ the matrix\linebreak $(b^1_z(x),\ldots,b^n_z(x))$ and $b_z:=b_z(1)$. The crucial observation for the definition of the conjugate index is given by the following lemma, which can be found in \cite{Pejsachowicz}. Nevertheless, we include its proof for the sake of completeness.

\begin{lemma}\label{binv}
$b_z$, $z=t+is\in\mathbb{C}$, is not invertible if and only if $s=0$ and $\ker\mathcal{A}_t\neq\{0\}$.
\end{lemma}

\begin{proof}
If $b_z$ is not invertible, take $0\neq v\in\ker b_z$ and define $w(x)=b_z(x)v$. Then $Jw''+S_tw+isw=0$, $w(0)=w(1)=0$, and by selfadjointness of $\mathcal{A}_t$ this is just possible if $s=0$.\\
On the other hand, if $Ju''+S_tu+isu=0$, $u(0)=u(1)=0$, it is easy to see that $0=u(1)=b_z(1)Ju'(0)$. Hence if $\det b_z\neq 0$, this implies $u'(0)=0$ and therefore $u=0$.  
\end{proof}

Let $c:S^1\rightarrow\mathbb{C}$ be any simple closed positively oriented path surrounding $I\times\{0\}\subset\mathbb{C}$ (i.e. having winding number $1$ with respect to any point in $I$). This induces a path $c_\gamma:S^1\rightarrow\mathbb{C}\setminus\{0\}$ by $c_\gamma(\varphi)=\det b_{c(\varphi)}$, and by definition the \textbf{conjugate index} of the geodesic is
\begin{align*}
\mu_{con}(\gamma)=w(c_\gamma,0),
\end{align*} 
where the latter denotes the winding number of $c_\gamma$ with respect to $0$.\\
With all this said we can state the Morse index theorem \cite{Pejsachowicz}.

\begin{theorem}\label{MIT}
Let $(M,g)$ be a semi-Riemannian manifold and let $\gamma:I\rightarrow M$ be a geodesic such that $1$ is not a conjugate instant. Then 
\begin{align*}
\mu_{spec}(\gamma)=\mu_{con}(\gamma).
\end{align*}
\end{theorem}

Finally, let us point out that a minor mistake appeared in the proof of this theorem in \cite{Pejsachowicz}. As a consequence, our definition of the conjugate index differs from the former one by sign. We will explain this more precisely in the first step of our proof below.

\section{The Proof}

\subsection{Preliminaries I:\quad K-Theory with Compact Supports}
The aim of this section is a brief summary of the basics of K-theory with compact supports. For a more detailed exposition we refer to \cite[section 3.2]{Fedosov}.\\ 
Let $X$ be a manifold. A virtual bundle with compact support is a triple $\xi=[E^0,E^1,a]$, where $E^0, E^1$ are vector bundles of the same dimension and $a:E^0\rightarrow E^1$ is a bundle isomorphism defined over $X\setminus\supp\xi$, where $\supp\xi\subset X$ is a compact subset called the support of $\xi$. Such an element is called trivial if $a$ can be extended to an isomorphism over the whole of $X$.  A sum is defined on the set of all virtual bundles with compact support on $X$ in the obvious way. Two triples $\xi_0=[E^0_0,E^1_0,a_0]$, $\xi_1=[E^0_1,E^1_1,a_1]$ are called isomorphic if there are bundle isomorphisms $\varphi^i:E^i_0\rightarrow E^i_1$, $i=0,1$, such that the diagram
\begin{align*}
\xymatrix{E^0_0\ar[r]^{a_0}\ar[d]_{\varphi^0}&E^1_0\ar[d]^{\varphi^1}\\
E^0_1\ar[r]^{a_1}&E^1_1
}
\end{align*}
is commutative over $X\setminus(\supp\{\xi_0\}\cup\supp\{\xi_1\})$. Finally, an equivalence relation can be introduced by defining $\xi_1\sim\xi_2$ if there are trivial virtual bundles $\eta_1,\eta_2$ such that $\xi_1+\eta_1$ is isomorphic to $\xi_2+\eta_2$. The equivalence classes form a group $K_c(X)$ with respect to the defined sum. Additionally, the following properties hold:
\begin{itemize}
\item If $a(t):E^0\rightarrow E^1$, $t\in[0,1]$, is a homotopy of bundle isomorphisms which are defined outside a fixed compact subset of $X$, then 
\begin{align*}
[E^0,E^1,a(0)]=[E^0,E^1,a(t)]\in K_c(X)\qquad\text{ for all }t\in[0,1].
\end{align*} 
\item If $\xi_0=[E^0,E^1,a_0],\; \xi_1=[E^1,E^2,a_1]\in K_c(X)$, their sum is given by
\begin{align*}
\xi_0+\xi_1=[E^0,E^2,a_1\circ a_0]\in K_c(X).
\end{align*}
This is called the logarithmic property of $K_c(X)$.
\end{itemize}

Finally, we will need the following essentially well-known assertion, whose proof is left to the reader.

\begin{lemma}\label{chern}
Let $[E_0,E_1,a]\in K_c(\mathbb{C})$ and $c:S^1\rightarrow\mathbb{C}$ be any simple closed positively oriented path surrounding  $\supp a$. Moreover, let $\psi: E_0\rightarrow\Theta(\mathbb{C}^n)$ and $\varphi:E_1\rightarrow\Theta(\mathbb{C}^n)$ be global trivializations. Then
\begin{align*}
c_1([E_0,E_1,a])=w(\det(\varphi\circ a\circ\psi^{-1})\circ c,0)\in\mathbb{Z},
\end{align*}
where $c_1:K_c(\mathbb{C})\rightarrow\mathbb{Z}$ denotes the first Chern number.
\end{lemma}

\subsection{Preliminaries II: On Spectral Flow}
Spectral flow is an integer valued homotopy invariant of paths of selfadjoint Fredholm operators introduced in \cite{AtiyahPatodi}. Here our basic reference is the paper \cite{Robbin-Salamon} by Robbin and Salamon, where spectral flow is defined for a certain class of unbounded selfadjoint Fredholm operators with fixed domain. Without repeating any details of their construction, we want to introduce spectral flow by means of a uniqueness result.\\
In the following, let $H$ be a real Hilbert space and let $W\subset H$ be a Hilbert space in its own right with a compact dense injection $W\hookrightarrow H$. We define a subspace $\mathcal{FS}(W,H)\subset\mathcal{L}(W,H)$ of the bounded operators with the norm topology consisting of those operators which are selfadjoint when regarded as unbounded operators in $H$ with dense domain $W$. Note that each element in $\mathcal{FS}(W,H)$ has a compact resolvent and hence is in particular a Fredholm operator.

\begin{theorem}[\cite{Robbin-Salamon},Theorem 4.3]\label{Robbin-Salamon Theorem}
There exist unique maps 
\begin{align*}
\sfl:\{\mathcal{A}:I\rightarrow\mathcal{FS}(W,H):\mathcal{A}\;\text{continuous},\mathcal{A}_0,\mathcal{A}_1\;\text{invertible}\}\rightarrow\mathbb{Z},
\end{align*}
one for every choice of $H$ and $W$ as above, satisfying the following axioms:
\begin{enumerate}
\item $\sfl$ is invariant under homotopies inside $\mathcal{FS}(W,H)$ through paths with invertible ends.
\item If $\mathcal{A}$ is constant, then $\sfl(\mathcal{A})=0$.
\item If $\mathcal{A}^1\oplus\mathcal{A}^2:W_1\oplus W_2\rightarrow H_1\oplus H_2$ is the pointwise direct sum of $\mathcal{A}^1\in\mathcal{FS}(W_1,H_1)$ and $\mathcal{A}^2\in\mathcal{FS}(W_2,H_2)$, then $\sfl(\mathcal{A}^1\oplus\mathcal{A}^2)=\sfl(\mathcal{A}^1)+\sfl(\mathcal{A}^2)$.
\item If the concatenation $\mathcal{A}^1\ast\mathcal{A}^2$ is defined, then $\sfl(\mathcal{A}^1\ast\mathcal{A}^2)=\sfl(\mathcal{A}^1)+\sfl(\mathcal{A}^2)$.
\item For $W=H=\mathbb{R}$ and $\mathcal{A}_t=\arctan(t-\frac{1}{2})$, we have $\sfl(\mathcal{A})=1$.
\end{enumerate}
\end{theorem}

\begin{rem}\label{remSF}
Condition (4) is in fact not necessary for the uniqueness, as shown in \cite{Robbin-Salamon}. 
\end{rem}

Below we will give an explicit construction of spectral flow which is different from the one in \cite{Robbin-Salamon}.
We begin by introducing a variant of the Atiyah-J\"anich bundle for families of Fredholm operators of index $0$, which is adapted from \cite{PejsachowiczII}.\\
Let $\Lambda$ be a manifold, let $\mathcal{X}$ be a Banach bundle over $\Lambda$ and let $Y$ be a Banach space. Moreover, let $L:\mathcal{X}\rightarrow\Theta(Y)$ be a bundle morphism that is fibrewise Fredholm of index $0$ and has compact support
\begin{align*}
\supp L=\{\lambda\in\Lambda:L_\lambda\notin GL(\mathcal{X}_\lambda,Y)\}.
\end{align*}
Here $GL$ stands for bounded invertible operators between fixed Banach spaces. Moreover, in the following we will denote the set of all such bundle morphisms by $\mathfrak{F}_{0,c}(\mathcal{X},\Theta(Y))$.\\ 
Using the compactness assumption, it is not difficult to show that there exists a finite dimensional subspace $V\subset Y$ such that
\begin{align*}
\im L_{\lambda}+V=Y\quad\forall\lambda\in\Lambda.
\end{align*}
We obtain a surjective bundle morphism
\begin{align*}
\mathcal{X}\xrightarrow{L}\Theta(Y)\rightarrow\Theta(Y/V)
\end{align*}
which yields by \cite[III,§3]{Lang} a subbundle $E(L,V)$ of $\mathcal{X}$ such that the total space is given by the kernel of this map. More precisely, the fibres of $E(L,V)$ can be described explicitely by
\begin{align*}
\{u\in\mathcal{X}_{\lambda}:L_{\lambda}u\in V\},\quad \lambda\in\Lambda.
\end{align*} 
We now define the index bundle of $L\in\mathfrak{F}_{0,c}(\mathcal{X},\Theta(Y))$ as the element
\begin{align*}
\ind(L)=[E(L,V),\Theta(V),L\mid_{E(L,V)}]\in K_c(\Lambda).
\end{align*}
Besides well definedness, the index bundle has the following properties:
\begin{itemize}
\item Let $L\in\mathfrak{F}_{0,c}(\mathcal{X},\Theta(Y))$ such that $L_\lambda\in GL(\mathcal{X}_\lambda,Y)$ for every $\lambda\in\Lambda$. Then $\ind(L)=0\in K_c(\Lambda)$.
\item If $H:I\times\mathcal{X}\rightarrow\Theta(Y)$ is a homotopy through bundle morphisms with compact support, then
\begin{align*}
\ind(H_0)=\ind(H_1).
\end{align*}
\item Let $L\in\mathfrak{F}_{0,c}(\mathcal{X},\Theta(Y))$ and $M\in\mathfrak{F}_{0,c}(\Theta(Y),\Theta(Z))$. Then
\begin{align*}
\ind(M\circ L)=\ind(M)+\ind(L).
\end{align*}
\end{itemize}
The proofs of these properties are based on the basic properties of $K$-theory with compact supports as stated in the previous section and are left to the reader.\\
Let $\mathcal{A}:I\rightarrow\mathcal{FS}(W,H)$ be a path with invertible ends. We extend $\mathcal{A}$ on $\mathbb{R}$ by setting $\mathcal{A}_t=\mathcal{A}_0$ for $t<0$ and $\mathcal{A}_t=\mathcal{A}_1$ for $t>1$. Moreover, if we consider the pointwise complexification $\mathcal{A}_t^{\mathbb{C}}$, we obtain a path in $\mathcal{FS}(W^\mathbb{C},H^\mathbb{C})$. Finally, we extend to a family on the complex plane by
\begin{align*}
\overline{\mathcal{A}^\mathbb{C}}_zu:=\mathcal{A}^\mathbb{C}_tu+is\cdot u,\quad z=t+is\in\mathbb{C}.
\end{align*}
Now we can regard $\overline{\mathcal{A}^\mathbb{C}}$ as a Banach bundle morphism between the trivial bundles $\Theta(W^\mathbb{C})$ and $\Theta(H^\mathbb{C})$ over $\mathbb{C}$. Moreover, due to the selfadjointness of $\mathcal{A}^\mathbb{C}_t$, $t\in\mathbb{R}$, when considered as unbounded operators on $H^\mathbb{C}$, we see that $\overline{\mathcal{A}^\mathbb{C}}_z\in GL(W^\mathbb{C},H^\mathbb{C})$ whenever $z=t+is\notin\mathbb{R}\times\{0\}$. Hence  $\overline{\mathcal{A}^\mathbb{C}}\in\mathfrak{F}_{0,c}(\Theta(W^\mathbb{C}),\Theta(H^\mathbb{C}))$ and the index bundle of $\overline{\mathcal{A}^\mathbb{C}}$ is defined.
 
\begin{lemma}\label{compspec}
The first Chern number of $\ind(\overline{\mathcal{A}^\mathbb{C}})\in K_c(\mathbb{C})$ coincides with the spectral flow of $\mathcal{A}$.
\end{lemma}

\begin{proof}
It is enough to check the assumptions of theorem \ref{Robbin-Salamon Theorem} (remember remark \ref{remSF}!).\\
At first, (2) and (3) follow immediately from the corresponding properties of the index bundle. Moreover, if $\mathcal{A}^1$ and $\mathcal{A}^2$ are homotopic by means of a homotopy $H:I\times I\rightarrow\mathcal{FS}(W,H)$ such that $H(\lambda,0)$ and $H(\lambda,1)$ are invertible for all $\lambda\in I$, we can build $\overline{H}^\mathbb{C}$ in the same way as described above for a single path of operators. We obtain a homotopy $\overline{H}^\mathbb{C}:I\times\Theta(W^\mathbb{C})\rightarrow\Theta(H^\mathbb{C})$ through bundle morphisms with compact support which is actually a homotopy between $\overline{(\mathcal{A}^1)^\mathbb{C}}$ and $\overline{(\mathcal{A}^2)^\mathbb{C}}$. This implies the first property of theorem \ref{Robbin-Salamon Theorem} by homotopy invariance of the index bundle.\\
Finally, if $W=H=\mathbb{R}$ and $\mathcal{A}_t=\arctan(t-\frac{1}{2})$, the index bundle of $\overline{\mathcal{A}^\mathbb{C}}$ is given by $[\Theta(\mathbb{C}),\Theta(\mathbb{C}),a]$, where $a$ is the map 
\begin{align*}
a:\Theta(\mathbb{C})\rightarrow\Theta(\mathbb{C}),\; (z,u)\mapsto (\arctan(t-\frac{1}{2})+is)u,\quad z=t+is\in\mathbb{C}.
\end{align*}
Now $c_1([\Theta(\mathbb{C}),\Theta(\mathbb{C}),a])=1$ follows immediately from lemma \ref{chern}.
\end{proof}

\subsection{Proof of Theorem \ref{MIT}} 
Let $\gamma:I\rightarrow M$ be a geodesic such that $1$ is not a conjugate instant. Moreover, let $L_t:H^1_0(I,\mathbb{R}^n)\rightarrow H^1_0(I,\mathbb{R}^n)$ and $\mathcal{A}_t:H^2(I,\mathbb{R}^n)\cap H^1_0(I,\mathbb{R}^n)\rightarrow L^2(I,\mathbb{R}^n)$ be the associated paths of operators as introduced in \eqref{RIESZDARSTELLUNG-2} and \eqref{JacobiDGL} respectively. Note that since $1$ is not a conjugate instant, $\mathcal{A}_0$ and $\mathcal{A}_1$ are invertible and hence the spectral flow $\sfl(\mathcal{A})$ is defined by theorem \ref{Robbin-Salamon Theorem}. 

\subsubsection*{Step 1: $\mu_{spec}(\gamma)=\sfl(\mathcal{A})$}
The first part of the proof coincides with the one in \cite{Pejsachowicz} and shows the equality of the spectral flow of $L$ and $\mathcal{A}$ up to sign. We will use the machinery of crossing forms as developed in \cite{Robbin-Salamon} and \cite{Specflow}, respectively\footnote{However, if the underlying metric $g$ is Riemannian, the spectral flow of $L$ is just minus the Morse index of $L_1$, and the claim follows from an easy integration by parts argument and the spectral decomposition of $\mathcal{A}_1$.}. Accordingly, we consider the perturbed paths $\mathcal{A}^\delta_t=\mathcal{A}_t+\delta id$ and $L_t^\delta$, where $\langle L^\delta_tu,u\rangle_{H^1_0(I,\mathbb{R}^n)}=q_t(u)-\delta\|u\|^2_{L^2(I,\mathbb{R}^n)}$, $u\in H^1_0(I,\mathbb{R}^n)$. Now we can find $\delta>0$ such that $\sfl(\mathcal{A}^\delta)=\sfl(\mathcal{A})$, $\sfl(L^\delta)=\sfl(L)$ and the so-called crossing forms defined as
\begin{align*}
\Gamma(\mathcal{A}^\delta,t)=\langle\dot {\mathcal{A}}^\delta_t\cdot,\cdot\rangle_{L^2(I,\mathbb{R}^n)}\mid_{\ker\mathcal{A}^\delta_t},\quad\Gamma(L^\delta,t)=\langle \dot{L}^\delta_t\cdot,\cdot\rangle_{H^1_0(I,\mathbb{R}^n)}\mid_{\ker L^\delta_t}
\end{align*}
are nondegenerate for all $t\in I$. Here $\dot{}$ means differentiation with respect to $t$ and the norm topology. Then the kernels of $\mathcal{A}_t^\delta$ and $L_t^\delta$ are nontrivial for only finitely many instants $t\in I$ and the spectral flows can be computed as
\begin{align*}
\sfl(\mathcal{A}^\delta)=\sum_{t\in I}{\sgn\Gamma(\mathcal{A}^\delta,t)}\quad\text{and}\quad\sfl(L^\delta)=\sum_{t\in I}{\sgn\Gamma(L^\delta,t)}
\end{align*}
respectively. Now we obtain from integration by parts
\begin{align}\label{intbypart}
\langle L^\delta_tu,v\rangle_{H^1_0(I,\mathbb{R}^n)}=-\langle \mathcal{A}^\delta_tu,v\rangle_{L^2(I,\mathbb{R}^n)},\quad u\in H^2(I,\mathbb{R}^n)\cap H^1_0(I,\mathbb{R}^n),\; v\in H^1_0(I,\mathbb{R}^n)
\end{align}
and therefore $\mu_{spec}(\gamma)=-\sfl(L)=\sfl(\mathcal{A})$.\\
Hence we have to compute the spectral flow of $\mathcal{A}$. We will use the result of lemma \ref{compspec} and concentrate in the following steps on the computation of the index bundle of $\overline{\mathcal{A}^\mathbb{C}}$.\\
In the proof of theorem \ref{MIT} in \cite{Pejsachowicz} the equality \eqref{intbypart} is used with the opposite sign on the right hand side. Hence also their final result and, accordingly, their definition of the conjugate index differs from ours by a sign.  

\subsubsection*{Step 2: Simplification of the Index Bundle I:\quad$\ind(\overline{\mathcal{A}^\mathbb{C}})=\ind(M)$} 
The aim of this step is to show that the index bundle of $\overline{\mathcal{A}^\mathbb{C}}$ is equal to the index bundle of the family
\begin{align*}
\begin{split}
&M_z:\mathcal{H}\rightarrow L^2(I,\mathbb{C}^{2n}),\quad z=t+is\in\mathbb{C},\\
&M_zw(x)=\sigma w'(x)+H_z(x)w(x),
\end{split}
\end{align*}
where $\mathcal{H}=\{w\in H^1(I,\mathbb{C}^{2n}):w(0),w(1)\in\{0\}\times\mathbb{C}^n\}$ and
\begin{align*}
H_z(x)=\begin{pmatrix}
-S_z(x)&0\\
0&-J
       \end{pmatrix},\quad S_z(x):=S_t(x)+isI,\,z=t+is\in\mathbb{C}.
\end{align*}
Here $\sigma$ denotes the symplectic matrix
\begin{align*}
\begin{pmatrix}
0&-id\\
id&0
     \end{pmatrix}.
\end{align*}
More precisely, $M_zw$, $w=(w_1,w_2)\in\mathcal{H}$, is given by
\begin{align*}
M_zw=\begin{pmatrix}
0&-id\\
id&0
     \end{pmatrix}\begin{pmatrix}
w'_1\\
w'_2
\end{pmatrix}+\begin{pmatrix}
-S_z&0\\
0&-J
\end{pmatrix}\begin{pmatrix}
w_1\\
w_2
\end{pmatrix}=
\begin{pmatrix}
-w'_2-S_zw_1\\
w'_1-Jw_2
\end{pmatrix}.
\end{align*}
We define operators $j:H^2(I,\mathbb{C}^n)\cap H^1_0(I,\mathbb{C}^n)\hookrightarrow\mathcal{H}$ and $\iota:L^2(I,\mathbb{C}^n)\hookrightarrow L^2(I,\mathbb{C}^{2n})$ by
\begin{align}\label{mapj}
j(u)=(u,Ju'),\qquad\iota(u)=(-u,0)
\end{align}
and obtain a diagram
\begin{align}\label{diagram2}
\xymatrix{\mathcal{H}\ar[r]^{M_z}&L^2(I,\mathbb{C}^{2n})\\
H^2(I,\mathbb{C}^n)\cap H^1_0(I,\mathbb{C}^n)\ar[u]^{j}\ar[r]^{\qquad\quad\overline{\mathcal{A}^\mathbb{C}}_z}&L^2(I,\mathbb{C}^n)\ar[u]_{\iota}
}
\end{align}
The diagram is commutative because
\begin{align*}
\iota(\overline{\mathcal{A}^\mathbb{C}}_zu)=(-Ju''(x)-S_z(x)u(x),0)=M_z(j(u))
\end{align*}
for $u\in H^2(I,\mathbb{C}^n)\cap H^1_0(I,\mathbb{C}^n)$, $z\in\mathbb{C}$. Now, let $V\subset L^2(I,\mathbb{C}^{n})$ be a finite dimensional subspace such that 
\begin{align*}
\im(\overline{\mathcal{A}^\mathbb{C}}_z)+V=L^2(I,\mathbb{C}^{n}),\; z\in\mathbb{C}.
\end{align*}
Setting $w_2=Jw'_1$, we obtain
\begin{align*}
\{(-Jw''_1-S_zw_1,0):w_1\in H^2(I,\mathbb{C}^n)\cap H^1_0(I,\mathbb{C}^n)\}+\iota(V)=L^2(I,\mathbb{C}^n)\oplus\{0\}
\end{align*}
and using $\{w'_1-Jw_2:(w_1,w_2)\in\mathcal{H}\}=L^2(I,\mathbb{C}^n)$, it is easy to see that
\begin{align*}
\im(M_z)+\iota(V)=L^2(I,\mathbb{C}^{2n}),\quad z\in\mathbb{C}.
\end{align*}
Hence $E(M,\iota(V))$ is defined. Moreover, by commutativity of \eqref{diagram2}, $j$ induces a bundle morphism $E(\overline{\mathcal{A}^\mathbb{C}},V)\rightarrow E(M,\iota(V))$. Taking into account that $j$ is injective and both bundles have the same fibre dimensions, it is clear that $j$ is actually a bundle isomorphism. We finally obtain a commutative diagram

\begin{align*}
\xymatrix{E(M,\iota(V))\ar[r]^{\quad M}&\Theta(\iota(V))\\
E(\overline{\mathcal{A}^\mathbb{C}},V)\ar[u]^{j}_{\cong}\ar[r]^{\qquad\overline{\mathcal{A}^\mathbb{C}}}&\Theta(V)\ar[u]^{\cong}_{\iota}
}
\end{align*}
showing $\ind(\overline{\mathcal{A}^\mathbb{C}})=\ind(M)$.

\subsubsection*{Step 3: Simplification of the Index Bundle II:\quad $\ind(M)=\ind(N)$}

We define a family of topological isomorphisms by
\begin{align*}
U:\mathbb{C}\rightarrow GL(L^2(I,\mathbb{C}^{2n})),\qquad (U_zw)(x)=\Psi_z(x)w(x),
\end{align*}
where $\Psi_z(x)$, $z\in\mathbb{C}$, is the solution of the initial value problem
\begin{align*}
\begin{cases}
\Psi'_z(x)=\sigma H_z(x)\Psi_z(x),\\
\Psi_z(0)=id.
\end{cases}
\end{align*}
It is easy to see that $\Psi_z(x)$ is symplectic; i.e. $\Psi_z^T(x)\sigma\Psi_z(x)=\sigma$, for all $x\in I$ and $z\in\mathbb{C}$. Here $\cdot^T$ denotes the transpose instead of the conjugate transpose.\\  
Moreover, we define a family
\begin{align*}
N_z:\tilde H_z\rightarrow L^2(I,\mathbb{C}^{2n}),\quad N_zw=U^T_zM_zU_z,\quad z\in\mathbb{C},
\end{align*}
where
\begin{align*}
\tilde H_z=\{w\in H^1(I,\mathbb{C}^{2n}):w(0)\in\{0\}\times\mathbb{C}^n,w(1)\in\Psi^{-1}_z(\{0\}\times\mathbb{C}^{n})\}
\end{align*}
and $\Psi_z:=\Psi_z(1)$. Note that $U_z(\tilde H_z)=\mathcal{H}$ for all $z\in\mathbb{C}$ such that the operators $N_z$ indeed are well defined. We obtain
\begin{align*}
(N_zw)(x)&=(U^T_zM_zU_zw)(x)=\Psi_z(x)^T(\sigma(\Psi_z(x)w(x))'+H_z(x)\Psi_z(x)w(x))\\
&=\Psi_z(x)^T(\sigma\Psi'_z(x)w(x)+\sigma\Psi_z(x)w'(x)+H_z(x)\Psi_z(x)w(x))\\
&=\Psi_z(x)^T(-H_z(x)\Psi_z(x)w(x)+\sigma\Psi_z(x)w'(x)+H_z(x)\Psi_z(x)w(x))\\
&=\Psi_z(x)^T\sigma\Psi_z(x)w'(x)=\sigma w'(x),\quad w\in\tilde H_z.
\end{align*}
Our next goal is to show that $\ind(M)=\ind(N)$, but here we have to handle a family with nonconstant domains. Fortunately, by the following result, the family of domains fit together to a Banach bundle $\tilde{\mathcal{H}}$ and $N$ can be regarded as a bundle morphism from $\tilde{\mathcal{H}}$ to $\Theta(L^2(I,\mathbb{C}^{2n}))$.

\begin{lemma}\label{bundleproj}
Let $\Lambda$ be a topological space, $A:\Lambda\rightarrow GL(m,\mathbb{C})$ a continuous family of invertible matrices and $U,V\subset\mathbb{C}^{m}$ fixed subspaces. Then
\begin{align*} 
\{(\lambda,u)\in\Lambda\times H^1(I,\mathbb{C}^{m}):u(0)\in U, u(1)\in A_\lambda V\}\subset\Lambda\times H^1(I,\mathbb{C}^{m})
\end{align*}
endowed with the induced topology is a Banach bundle.
\end{lemma}

\begin{proof}
Let $\tilde P:\Lambda\rightarrow M(m,\mathbb{C})$ be the family of orthogonal projections in $\mathbb{C}^{m}$ such that $\im\tilde P_\lambda= A_\lambda V$, $\lambda\in\Lambda$. Define 
\begin{align*}
P:\Lambda\rightarrow\mathcal{L}(H^1(I,\mathbb{C}^{m})),\quad (P_\lambda w)(x)=w(x)-(1-x)\pr_{U^\perp}(w(0))-x(id-\tilde P_\lambda)w(1),
\end{align*}
where $pr_{U^\perp}:\mathbb{C}^{m}\rightarrow\mathbb{C}^{m}$ denotes the orthogonal projection onto $U^\perp$. It is easy to show that $P_\lambda$ is a projection onto $\{u\in H^1(I,\mathbb{C}^{m}):u(0)\in U, u(1)\in A_\lambda V\}$.\\ 
Now the claim follows from \cite[Proposition 3.2.1]{FiPejsachowicz}, where it is proven that, given a continuous family $P_\lambda$ of projections in a Banach space $X$, the set $\{(\lambda,u)\in\Lambda\times X:P_\lambda u=u\}\subset\Lambda\times X$ is a Banach subbundle.
\end{proof}

Since $\tilde{\mathcal{H}}$ carries the subspace topology of $\Lambda\times H^1(I,\mathbb{C}^{2n})$, it is clear that the restriction of $U$ defines a bundle morphism from $\tilde{\mathcal{H}}$ to $\Theta(\mathcal{H})$. Hence we can decompose $N$ as the following sequence of bundle morphisms:
\begin{align*}
\tilde{\mathcal{H}}\xrightarrow{U} \Theta(\mathcal{H})\xrightarrow{M} \Theta(L^2(I,\mathbb{C}^{2n}))\xrightarrow{U^T} \Theta(L^2(I,\mathbb{C}^{2n})).
\end{align*} 
Using the properties of the index bundle stated in section 3.2, we finally compute
\begin{align*}
\ind(N)&=\ind(U^TMU)=\ind(U^T)+\ind(M)+\ind(U)=\ind(M).
\end{align*}
The family $N$ is simple enough to compute its index bundle explicitely, which is the subject of the following final step.

\subsubsection*{Step 4: Proof of the Theorem: $c_1(\ind(N))=\mu_{con}(\gamma)$}
If we set $Y_1=\mathbb{C}^{2n}\subset L^2(I,\mathbb{C}^{2n})$, the subspace of constant functions, and 
\begin{align*}
Y_2=\{y\in L^2(I,\mathbb{C}^{2n}):\int^1_0{y(x)dx}=0\},
\end{align*}
we have a direct sum decomposition $L^2(I,\mathbb{C}^{2n})=Y_1\oplus Y_2$.\\
Now, for any fixed $y\in Y_2$, we define $F\in H^1(I,\mathbb{C}^{2n})$ by
\begin{align*}
F(x)=-\sigma\int^x_0{y(s)ds}.
\end{align*}
Since $F(0)=0$ and $F(1)=0\in\Psi^{-1}_z(\{0\}\times\mathbb{C}^n)$, we conclude $F\in\tilde H_z$ and, moreover, $N_zF=y$, $z\in\mathbb{C}$. Hence $N_z(\tilde H_z)\supset Y_2$, showing that
\begin{align*}
\im(N_z)+Y_1=L^2(I,\mathbb{C}^{2n})\quad\text{ for all } z\in\mathbb{C}.
\end{align*}
We obtain
\begin{align*}
\ind(N)=[E(N,Y_1),\Theta(Y_1),N]\in K_c(\mathbb{C}),
\end{align*}
where the total space of the bundle $E(N,Y_1)$ is given by
\begin{align*}
&\{(z,w)\in\mathbb{C}\times H^1(I,\mathbb{C}^{2n}):w\in\tilde H_z,\sigma w'\in Y_1\}\\
&=\{(z,w)\in\mathbb{C}\times H^1(I,\mathbb{C}^{2n}):w\in\tilde H_z, w'\equiv const.\}\\
&=\{(z,w)\in\mathbb{C}\times H^1(I,\mathbb{C}^{2n}):w(x)=(1-x)a+xb,a\in\{0\}\times\mathbb{C}^n,b\in\Psi^{-1}_z(\{0\}\times\mathbb{C}^n)\}.
\end{align*}

With the isomorphisms
\begin{align*}
\Phi_1&:\Theta(Y_1)\rightarrow\Theta(\mathbb{C}^{2n}),\qquad\qquad \Phi_1(z,u)=(z,u(0)),\\
\Phi_2&:E(N,Y_1)\rightarrow\Theta(\mathbb{C}^{2n}),\quad\Phi_2(z,w)=(z,\pr_2(w(0)),\pr_2(\Psi_zw(1)))
\end{align*}
and the morphism
\begin{align*}
\tilde{N}:\Theta(\mathbb{C}^{2n})\rightarrow\Theta(\mathbb{C}^{2n}),\quad\tilde{N}(z,a,b)=\sigma(\Psi^{-1}_z(0,b)-(0,a)),
\end{align*}
we have a commutative diagram
\begin{align*}
\xymatrix{E(N,Y_1)\ar[d]^{\cong}_{\Phi_2}\ar[r]^{\quad N}&\Theta(Y_1)\ar[d]^{\Phi_1}_{\cong}\\
\Theta(\mathbb{C}^{2n})\ar[r]^{\tilde{N}}&\Theta(\mathbb{C}^{2n})
}
\end{align*}

This implies

\begin{align*}
[E(N,Y_1),\Theta(Y_1),N]=[\Theta(\mathbb{C}^{2n}),\Theta(\mathbb{C}^{2n}),\tilde{N}]\in K_c(\mathbb{C}).
\end{align*}

Note that the map $j(u)=(u,Ju')$, already introduced in \eqref{mapj}, is a bijection between the spaces of solutions of the differential equations $Ju''(x)+S_z(x)u(x)=0$ and $w'(x)=\sigma H_z(x)w(x)$, $z\in\mathbb{C}$. Hence $\Psi_z$ is of the form
\begin{align*}
\begin{pmatrix}
\ast & b_z\\
\ast & \ast
\end{pmatrix},
\end{align*}

where $b_z$ is the matrix family from the definition of the conjugate index.\\ 
Now $\tilde{N}$ is given by

\begin{align*}
\begin{split}
\tilde{N}(z,a,b)&=\sigma\left(\Psi^{-1}_z\begin{pmatrix}
0\\
b
                                                        \end{pmatrix}
-\begin{pmatrix}
0\\
a
 \end{pmatrix}\right)
=\sigma\left(-\sigma\Psi^T_z\sigma \begin{pmatrix}
0\\
b
                                                        \end{pmatrix}
-\begin{pmatrix}
0\\
a
 \end{pmatrix}\right)\\
&=\Psi^T_z\begin{pmatrix}
-b\\
0
 \end{pmatrix}+\begin{pmatrix}
a\\
0
 \end{pmatrix}=
\begin{pmatrix}
id&\ast \\
0&-b^T_z
\end{pmatrix}\begin{pmatrix}
a\\
b
\end{pmatrix},\quad z\in\mathbb{C},\; (a,b)\in\mathbb{C}^{2n},
\end{split}
\end{align*}

and using homotopy invariance of the index bundle, we get

\begin{align*}
[\Theta(\mathbb{C}^{2n}),\Theta(\mathbb{C}^{2n}),\tilde{N}]&=[\Theta(\mathbb{C}^{n})\oplus\Theta(\mathbb{C}^n),\Theta(\mathbb{C}^{n})\oplus\Theta(\mathbb{C}^n),id\oplus(-b^T)]\\
&=[\Theta(\mathbb{C}^{n})\oplus\Theta(\mathbb{C}^n),\Theta(\mathbb{C}^{n})\oplus\Theta(\mathbb{C}^n),id\oplus b^T]\\
&=[\Theta(\mathbb{C}^{n}),\Theta(\mathbb{C}^{n}),b^T].
\end{align*}

Finally, using lemma \ref{chern} and the invariance of the determinant under transposition, we obtain $c_1(\ind(N))=\mu_{con}(\gamma)$, which completes the proof of the Morse index theorem, theorem \ref{MIT}.

\vspace{1cm}

GEORG-AUGUST-UNIVERSIT\"AT G\"OTTINGEN\\
MATHEMATISCHES INSTITUT\\
BUNSENSTRA{\ss}E 3-5\\
D-37073 G\"OTTINGEN\\
GERMANY\\
E-MAIL ADDRESS: waterstraat@daad-alumni.de
\end{document}